\documentclass[12pt]{amsart}

\usepackage{amssymb}
\usepackage{amsfonts}
\usepackage{amsmath}
\usepackage[all]{xy}

\usepackage[margin=3cm]{geometry}

\numberwithin{equation}{section}

\newtheorem{question}[equation]{Question}
\newtheorem{thm}[equation]{Theorem}

\newtheorem{prop}[equation]{Proposition}

\theoremstyle{definition}
\newtheorem{dfn}[equation]{Definition}
\newtheorem{exam}[equation]{Example}

\theoremstyle{remark}
\newtheorem{rem}[equation]{Remark}

\renewcommand{\leq}{\leqslant}
\renewcommand{\geq}{\geqslant}

\newcommand{\abs}[1]{\left\lvert#1\right\rvert}

\newcommand{\norm}[1]{\left\|#1\right\|}
\newcommand{\MA}{\mathop{\mathrm{MA}}\nolimits}
\newcommand{\Proj}{\mathop{\mathrm{Proj}}\nolimits}

\newcommand{\PSH}{\mathop{\mathrm{PSH}}\nolimits}

\newcommand{\vol}{\mathop{\mathrm{vol}}\nolimits}

\begin{document}

\title{On the volume of graded linear series and Monge--Amp\`{e}re mass}
\author{Tomoyuki Hisamoto}
\address{Graduate School of Mathematical Sciences, The University of Tokyo, 
3-8-1 Komaba Meguro-ku, Tokyo 153-0041, Japan}
\email{hisamoto@ms.u-tokyo.ac.jp}
\date{}
\maketitle

\begin{abstract}

We give an analytic description of the volume of a graded linear series, as the Monge--Amp\`{e}re mass of a certain equilibrium metric associated to any smooth Hermitian metric on the line bundle. We also show the continuity of this equilibrium metric on some Zariski open subset, under a geometric assumption. 

\end{abstract}

\section{Introduction} 
Let $L$ be a holomorphic line bundle on a smooth projective variety $X$. A {\em graded linear series} of $L$ is a graded subalgebra $W$ of the section ring $R=\bigoplus_{k\geq0} H^0(X, L^{\otimes k})$. In this paper we study the following invariant which plays a fundamental role in the asymptotic analysis of graded linear series. 
\begin{dfn}\label{volume}
We denote the dimension of $X$ by $n$. The {\em volume} of a graded linear series $W$ is the nonnegative real number defined by 
\begin{equation*}
\vol(W):=\limsup_{k \to \infty} \frac{\dim W_k}{k^n/n!}. 
\end{equation*}
 
\end{dfn}
This is known to be finite. The limit of supremum is in fact limit for sufficiently divisible $k$ from the result of \cite{KK09}. In case $W$ is complete, {\em i.e.} $W=R$, $\vol(R)$ is nothing but the volume of the line bundle $\vol(L)$ which is a birational invariant of $L$ and widely studied. We refer to Chapter $2$ of \cite{Laz04} for the basic facts. Analytic studies of the volume of line bundles were initiated by \cite{Bou02} and \cite{Ber09}. For a general graded linear series, \cite{KK09} and \cite{LM09} originally systematically studied properties of the volume by relating it with the {\em Okounkov body} of $W$. Proper subalgebras of $R$ naturally arise in many interesting situations of algebraic geometry (see {\em e.g.}  Example \ref{restricted} or \ref{ideal} below) and it is necessary to develop asymptotic analysis of graded linear series. 

In this paper we study the volume for a general graded linear series from the analytic point of view, succeeding to the spirit of previous work of Boucksom, Berman and many other authors. We work over the complex number field $\mathbb{C}$. The main result is an integral representation of the volume via the Monge--Amp\`{e}re product of a certain singular Hermitian metric on $L$, called the {\em equilibrium metric}, which is determined by $W$ with any fixed smooth Hermitian metric on $L$. 
In what follows we identify a singular Hermitian metric $h$ and its local weight function $\varphi$. They are related by the identity $h=e^{-\varphi}$ which holds in each local trivialization patch of $L$. We also identify a holomorphic section $\sigma \in H^0(X, L^{\otimes k})$ with the corresponding function on each local trivialization patch and denote by $\abs{\sigma}$ the absolute value taken for the function. For detail, see Section $2$. 
\begin{dfn}\label{ELW}
Let $h=e^{-\varphi}$ be a smooth Hermitian metric on $L$. For each $k \geq 1$, we define a singular Hermitian metric $h_k=e^{-\varphi_k}$ by  
\begin{align*}
\varphi_k := \sup \bigg\{ \frac{1}{k} \log\abs{\sigma}^2  \ \bigg| \  \sigma \in W_k,  \ \sup_X \abs{\sigma}^2e^{-k\varphi} \leq 1. \bigg\}. 
\end{align*}
The {\em equilibrium metric} of $h$ with respect to $W$ is the (possibly) singular Hermitian metric defined by its weight: 
\begin{align*}
P_W\varphi := (\sup_k \varphi_k)^*. 
\end{align*}
Here we denote the upper-semicontinuous envelope of a function $f$ by $f^*(x):=\limsup_{y \to x} f(y)$. 
\end{dfn}
The formulation of equilibrium metric here is originated from Siciak (\cite{Sic62}) and \cite{Ber09} introduced the corresponding idea for line bundles to investigate the asymptotic of related Bergman kernels. 
Roughly speaking, $P_{W}\varphi$'s epigraph is the $W$-polynomially convex hull of $\varphi$'s epigraph. $P_{W}\varphi$ is a plurisubharmonic function on each trivialization patch provided $W$ is non-trivial. That is, the curvature current $dd^c P_W\varphi$ defines a closed positive current on $X$. On the bounded locus of $P_W\varphi$ one can define the Monge--Amp\`{e}re product $(dd^cP_W\varphi)^n$ in the manner of Bedford--Taylor. Further, the trivial extension of $(dd^cP_W\varphi)^n$ defines a positive measure $\MA(P_W\varphi)$ which has no mass on any pluripolar subset of $X$. This kind of measures is called {\em non-pluripolar Monge--Amp\`{e}re product} and studied by \cite{BEGZ10}. In particular, it was proved that $\MA(P_W\varphi)$ has finite mass over $X$. The following is our main theorem. 
\begin{thm}\label{main}
Let $X$ be a smooth projective variety and $L$ a holomorphic line bundle on $X$. Let $W$ be a graded linear series of $L$ such that the associated map $X \dashrightarrow \mathbb{P}W_k^*$ is birational onto its image for any sufficiently divisible $k$. Then, for any smooth Hermitian metric $h=e^{-\varphi}$ on $L$, the Monge--Amp\`{e}re mass of the equilibrium metric $P_W\varphi$ gives the volume of $W$. That is, 
\begin{equation*}
\vol(W)=\int_X \MA(P_W\varphi)
\end{equation*}
holds. 
\end{thm}
This formula enables us to investigate the positivity of $W$ locally in $X$ and it might be helpful for lower bound estimate of $\vol(W)$ in the future. Theorem \ref{main} is a natural generalization of the results in \cite{Bou02}, \cite{Ber09}, \cite{BB10}, \cite{BEGZ10}, and \cite{His12}. Note that the birationality assumption is necessary for general $W\neq R$ (see Remark \ref{birational2}). Moreover, thanks to the currently available technology, the proof in the present paper get much simpler than the previous ones, even in the case of complete graded linear series. On the other hand, our approach is rather algebraic relying on the result of \cite{KK09}, \cite{LM09}, \cite{Jow10}, and \cite{DBP12} and does not give much information about Bergman kernel asymptotics as \cite{Ber09} or \cite{His12}. In particular, the following regularity problem is open. 
\begin{question}\label{regularity}
In the situation of Theorem \ref{main}, $P_W\varphi$ has Lipschitz continuous derivatives on some non-empty Zariski open subset of $X$?
\end{question}
In the complete case $W=R$ this is proved by \cite{Ber09}. It also holds for the {\em restricted} linear series, by \cite{His12}. See also \cite{BD09}. Such a regularity is the key to analyze the asymptotic of Bergman kernels. For a general graded linear series, we may show the following at the present. 
\begin{thm}\label{continuity}
In the situation of Theorem \ref{main} let us further 
assume that $W$ is finitely generated and $\Proj W$, which gives the image of the associated rational map $X \dashrightarrow \mathbb{P}W_k^*$ for $k$ sufficiently divisible, is normal. Then $P_W\varphi$ is continuous on some non-empty Zariski open subset of $X$. 
\end{thm}

The organization of this paper is as follows. In Section $2$ and $3$ we will give some preliminary materials, from the analytic and the algebraic viewpoints. We prove Theorem \ref{main} in Section $4$. Section $5$ will be devoted to the proof of Theorem \ref{continuity}. 

\section{Monge--Amp\`{e}re operator} 

In this section, we briefly review the definition and basic properties of the Monge--Amp\`{e}re operator. 

Let $L$ be a holomorphic line bundle on a projective manifold $X$. We usually fix a family of local trivialization patches $U_{\alpha}$ which cover $X$. A singular Hermitian metric $h$ on $L$ is by definition a family of functions $h_{\alpha}=e^{-\varphi_{\alpha}}$ which are defined on corresponding $U_{\alpha}$ and satisfy the transition rule: $\varphi_{\beta}=\varphi_{\alpha}-\log \abs{g_{\alpha \beta}}^2$ on $U_{\alpha}\cap U_{\beta}$. Here $g_{\alpha \beta}$ are the transition functions of $L$ with respect to the indices $\alpha$ and $\beta$. {\em The weight functions} $\varphi_{\alpha}$ are assumed to be locally integrable. If $\varphi_{\alpha}$ are smooth, $\{e^{-\varphi_{\alpha}}\}_{\alpha}$ defines a smooth Hermitian metric on $L$. We usually denote the family $\{ \varphi_{\alpha} \}_{\alpha}$ by $\varphi$ and omit the indices of local trivializations. Notice that each $\varphi=\varphi_{\alpha}$ is only a local function and not globally defined, but the curvature current $\Theta_h=dd^c \varphi$ is globally defined and is semipositive if and only if each $\varphi$ is plurisubharmonic ({\em psh} for short). Here we denote by $d^c$ the real differential operator $\frac{\partial-\bar{\partial}}{4\pi\sqrt{-1}}$.  We call such a weight a {\em psh weight}. The most important example is those of the form $k^{-1}\log (\abs{\sigma_1}^2+\cdots +\abs{\sigma_N}^2)$, defined by some holomorphic sections $\sigma_1, \cdots, \sigma_N  \in H^0(X, L^{\otimes k})$. Here $\abs{\sigma_i}$ ($1\leq i \leq N$) denotes the absolute value of the corresponding function of each $\sigma_i$ on $U_{\alpha}$. We call such weights {\em algebraic singular}. More generally, a psh weight $\varphi$ is said to have a {\em small unbounded locus} if it is locally bounded outside a closed complete pluripolar subset $S \subset X$. 
A singular Hermitian metric $h=e^{-\varphi}$ is said to have strictly positive curvature if $dd^c \varphi \geq \omega$ holds for some K\"{a}hler form $\omega$. 

Let $n$ be the dimension of $X$. The Monge--Amp\`{e}re operator is defined by   
\begin{equation*}
   \varphi \mapsto  \MA(\varphi) := (dd^c \varphi)^n 
\end{equation*}
when $\varphi$ is smooth. On the other hand it does not make sense for general $\varphi$. 
The celebrated result of Bedford--Taylor \cite{BT76} tells us 
that the right hand side can be defined as a current 
if $\varphi$ is at least in the class $L^{\infty}\cap \PSH(U_{\alpha})$. 
Specifically, by induction on the exponent $q = 1,2,...,n$, it can be defined as: 
\begin{equation*}
   \int_{U_{\alpha}} (dd^c \varphi)^q \wedge \eta :=  
   \int_{U_{\alpha}} \varphi (dd^c \varphi)^{q-1} \wedge dd^c \eta   
\end{equation*}
for each test form $\eta$.  
Here $ \int$ denotes the canonical pairing of currents  and test forms. 
This is indeed well-defined and defines a closed positive current, 
because $ \varphi $ is a bounded Borel function 
and $ (dd^c \varphi)^{q-1} $ has measure coefficients by the induction hypothesis. Notice the fact that any closed positive current has measure coefficients. 
Bedford--Taylor's Monge--Amp\`{e}re products have useful continuity properties:  
\begin{thm}[\cite{Ko05}, Theorem $1.11$, Proposition $1.12$, Theorem $1.15$.]\label{continuity properties of BT}
For any sequence of bounded psh weights, the convergence of Monge--Amp\`{e}re products 
  \begin{equation*}
  ( dd^c \varphi_k )^n \to (dd^c \varphi )^n 
  \end{equation*} 
  holds in the sense of currents if it satisfies one of the following conditions. 
    \begin{itemize}
    \setlength{\itemsep}{0pt}
     \item[$(1)$]
        $\varphi_k$ is non-increasing and converges to $\varphi$ pointwise in $X$.    
     \item[$(2)$]
        $\varphi_k$ is non-decreasing and converges to $\varphi$ almost everywhere in $X$. 
     \item[$(3)$] 
        $\varphi_k$ converges to $\varphi$ uniformly on any compact subset of $X$. 
   \end{itemize}
\end{thm}

It is also necessary to consider unbounded psh weights. 
On the other hand, for our purpose, 
it is enough to deal with weights with small unbounded loci. 

\begin{dfn}\label{MA}
 Let $\varphi$ be a psh weight of a singular metric on $L$. If $\varphi$ has a small unbounded locus contained in an algebraic subset $S$, we define a positive measure $\MA(\varphi)$ on $X$ by 
 \begin{equation*}
 \MA (\varphi)  := 
 \text{ the zero extension of } ( dd^c \varphi)^n . 
 \end{equation*}
 Note that the coefficient of $( dd^c \varphi )^n $ is well-defined 
 as a measure on $X \setminus S$. 
 $\hfill \Box$
\end{dfn}

Actually $(dd^c \varphi)^n$ has a finite mass so that $\MA(\varphi)$ defines a closed positive current on $X$. For a proof, see \cite{BEGZ10}, Section 1. 

\begin{rem}
In \cite{BEGZ10}, the {\em non-pluripolar} Monge--Amp\`{e}re product was defined 
in fact for general psh weights on a compact K\"{a}hler manifold.
Note that this definition of the Monge--Amp\`{e}re operator makes 
the measure $\MA(\varphi)$ to have no mass on any pluripolar set.  
Roughly speaking, $\MA(\varphi)$ ignores the mass which comes from the singularities of $\varphi$. 
For this reason, as a measure-valued function in $\varphi$, $\MA(\varphi)$ is no longer continuous. This also applies to $\MA(P_W\varphi)$ and one of the technical point for us to prove Theorem \ref{main}. 
$\hfill \Box$
\end{rem}
We recall the fundamental fact established in \cite{BEGZ10} 
which states that 
the less singular psh weight has the larger Monge--Amp\`{e}re mass. 
Recall that given two psh weight $\varphi$ and $\varphi'$ on $L$, 
$\varphi$ is said to be less singular than $\varphi'$ 
if there exists a constant $C > 0$ such that $\varphi' \leq \varphi + C$ holds on $X$. 
We say that a psh weight is {\em minimal singular} 
if it is minimal with respect to this partial order. 
When $\varphi$ is less singular than $\varphi'$ and $\varphi'$ is less singular than $\varphi$, 
we say that the two functions have the equivalent singularities. 
This defines a equivalence relation. 

\begin{thm}[\cite{BEGZ10}, Theorem 1.16.]\label{comparison theorem} 
Let $\varphi$ and $\varphi'$ be psh weights 
with small unbounded loci such that
$\varphi$ is less singular than $\varphi'$. Then 
\begin{equation*}
\int_X \MA(\varphi')
\leq 
\int_X \MA(\varphi)
\end{equation*}
holds. 
\end{thm}

\begin{rem}
  It is unknown whether Theorem \ref{comparison theorem} holds for general psh weights. 
  $\hfill \Box$
\end{rem}


\section{The volume of a graded linear series}

First we introduce the suitable class of graded linear series to handle with the  volume. Let $W=\bigoplus_{k\geq 0}W_k$ be a graded linear series. For each $k$, $W_k$ defines a rational map $f_k\colon X \dashrightarrow \mathbb{P}W_k^*$ unless $W_k=\{0\}$. Here $W_k^*$ denotes the dual vector space of $W_k$. Notice that the image of a rational map is defined to be the projection of the graph.  
\begin{dfn}
We call a graded linear series $W$ {\em birational} if the associated map $f_k\colon X \dashrightarrow \mathbb{P}W_k^*$ is birational onto its image for any sufficiently divisible $k$. 
\end{dfn}
This notion is introduced by \cite{LM09} (see Definition $2.5$). 
If the line bundle $L$ is big, the graded linear series $R=\bigoplus_{k\geq0}H^0(X, L^{\otimes k})$ is complete and by definition is birational. Thus the notion of birational graded linear series is a natural generalization of the complete linear series of a big line bundle. A line bundle $L$ is known to be big if and only if its volume is positive, {\em i.e.} $\vol(L)>0$. This is why we concentrate on the class of big line bundles in the study of the volume. It is also true that the volume of a birational graded linear series is positive (see \cite{LM09}, Lemma $2.6$ and Theorem $2.13$), but the converse does not hold. 

\begin{exam}\label{finite}
Let $f\colon X \to Y$ be a finite morphism to a projective variety and fix an embedding $Y \hookrightarrow \mathbb{P}^N$. Then $W_k:=f^*H^0(Y, \mathcal{O}_Y(k))$ defines a graded linear series with positive volume. But this is not birational unless $\deg f =1$.  
\end{exam}

The difference seems subtle at first glance, however in fact the class of graded linear series with positive volume is much harder to treat than the class of birational graded linear series. See Remark \ref{birational1}. 
To emphasize the significance of the generalization to non-complete linear series, let us describe some examples of birational graded linear series. 
\begin{exam}\label{restricted}
Let $Y$ be a smooth projective variety and $L$ a holomorphic line bundle defined over $Y$. We assume that $X$ is a closed subvariety of $Y$. The family of subspaces defined by 
\begin{equation*}
W_k:=\mathrm{Im} \big[ \ H^0(Y, L^{\otimes k}) \to H^0(X, {L^{\otimes k}|_X}) \ \big]
\end{equation*}
is called the {\em restricted linear series}. 
The volume $\vol_{Y|X}(L):=\vol(W)$ is called the {\em restricted volume}. If $X$ is not contained in the augmented base locus $\mathbb{B}_+(L)$ of $Y$ (for the definition, see \cite{Laz04}, Definition 10.3.2), $W$ defines a birational graded linear series. 
For restricted linear series, the author investigated the asymptotic of related Bergman kernels in \cite{His12} and obtained Theorem \ref{main} as a corollary. The point is that in this special case we can prove Question \ref{regularity} thanks to an $L^2$-extension theorem for the subvariety $X$. 
\end{exam}

\begin{exam}\label{ideal}
Let $\mathfrak{a}\subseteq \mathcal{O}_X$ be an ideal sheaf. The family of subspaces 
\begin{equation*}
W_k:= H^0(X, L^{\otimes k} \otimes \overline{\mathfrak{a}^k})
\end{equation*}
defines a graded linear series. Here $\overline{\mathfrak{a}^k}$ denotes the integral closure of $\mathfrak{a}^k$. There exist a modification $\mu \colon X' \to X$ and an effective divisor $F$ on $X'$ such that $\mathfrak{a}\cdot\mathcal{O}_{X'} = \mathcal{O}_{X'}(-F)$ and by definition $\overline{\mathfrak{a}}=\mu_*\mathcal{O}_{X'}(-F)$ hold. Hence we have $W_k = \mu_*H^0(X', \mu^*L^{\otimes k} \otimes \mathcal{O}_{X'}(-kF))$. In particular, $W$ is birational if and only if $\mu^*L \otimes \mathcal{O}_{X'}(-F)$ is big. If we assume further $\mu^*L \otimes \mathcal{O}_{X'}(-F)$ is ample, Theorem \ref{main} follows from the results of \cite{Ber07}, Section $4$. It seems not so hard to extend these results to the case when $\mu^*L \otimes \mathcal{O}_{X'}(-F)$ is big, along the same line as \cite{Ber07}. 
\end{exam}

\begin{exam}
Let $X:=\mathbb{P}^2$, $L:=\mathcal{O}(1)$, and $W:=\mathbb{C}[X, Y, YZ, YZ^2, \dots, YZ^i, \dots]$ $(i \geq 1)$ be the graded subalgebra of the homogeneous coordinate ring $\mathbb{C}[X, Y, Z]$. It is then easy to see that $W_k$ is base point free and the natural map $X \to \mathbb{P}W_k^*$ is birational onto its image for each $k$. However, $W$ is not finitely generated $\mathbb{C}$-algebra. 
\end{exam}

Study of the volume of birational graded linear series was first taken by \cite{KK09} and \cite{LM09}. They used the theory of Okounkov bodies to derive the log-concavity of the volume and the following type of Fujita's approximate Zariski decomposition. It motivates our study and will play the central role in the algebraic part of the proof of Theorem \ref{main}. See also \cite{DBP12} and \cite{Jow10}. 

\begin{thm}[\cite{KK09}, Theorem $5$. See also \cite{LM09}, Theorem 3.5 and \cite{DBP12}, Theorem 3.14.]\label{Fujita} 
Let $W$ be a graded linear series with $\vol(W)>0$. Then for any $\varepsilon>0$ there exists a number $\ell_0$ such that 
\begin{equation*}
\lim_{k\to \infty}\frac{\dim \mathrm{Im}\big[S^kW_{\ell}\to W_{k\ell}\big]}{k^n\ell^n/n!} \geq \vol(W)-\varepsilon 
\end{equation*} 
holds for any $\ell\geq\ell_0$. 
\end{thm}
We will use the following consequence which is equivalent to Theorem $\mathrm{C}$ of \cite{Jow10}, to prove Theorem \ref{main}. 
\begin{prop}\label{algebraic part}
Let $W$ be a birational graded linear series. For each $k$ let $\mu_k\colon X_k \to X$ be a resolution of the base ideal $\mathfrak{b}(W_k)$ such that $\mu_k^{-1}\mathfrak{b}(W_k)=\mathcal{O}(-F_k)$ holds for some effective divisor $F_k$ on $X_k$. Define the line bundle on $X_k$ by $M_k:= \mu_k^*L^{\otimes k} \otimes \mathcal{O}(-F_k)$ and denote the self-intersection number of the globally generated line bundle $M_k$ by $M_k^n$. Then it holds that 
\begin{equation*}
\vol(W)=\lim_{k \to \infty} \frac{M_k^n}{k^n}.  
\end{equation*}
\end{prop}
Here $k$ runs through sufficiently divisible numbers. 
\begin{proof}
Let us first assume that $W$ is finitely generated. In that case there exists some $\ell$ such that for each $k$ the natural map $S^kW_{\ell} \to W_{k\ell}$ is surjective. Then composing with inverse of the Segre embedding to the image, the induced $\mathbb{P}W_{k\ell}^* \to \mathbb{P}S^kW_{\ell}^*$ maps the image of $X \dashrightarrow \mathbb{P}W_{k\ell}^*$ onto the image of $X \dashrightarrow\mathbb{P}W_{\ell}^*$ isomorphically. Let us fix such $\ell$, and denote the image of the natural map $f\colon X \dashrightarrow \mathbb{P}W_{\ell}^*$ by $Y$. Notice that $Y$ is isomorphic to $\Proj W$, which is well-defined since $W$ is finitely generated. We also denote the image of the natural map $g_k \colon X \dashrightarrow \mathbb{P}H^0(X_k, M_k)^*$ by $Z_k$. Then for any $k$ divided by $\ell$ the linear projection $\mathbb{P}H^0(X_k, M_k)^* \dashrightarrow \mathbb{P}W_{k}^*$ composed with the above isomorphism between the images induces the morphism $\pi_k\colon Z_k \to Y$ such that the following diagram commutes. 
\[\xymatrix{
   {X_k}\ar[r]^{g_k}\ar[d]_{\mu_k} 
 & {Z_k}\ar[d]^{\pi_k} \\ 
   {X}\ar @{-->}[r]_{f} 
 & Y 
}
\]
Here $\pi_k \circ g_k$ gives a resolution of $f$ for $X_k$. By the definition of $\pi_k$, we have $\pi_k^*\mathcal{O}_Y(\frac{k}{\ell})=\mathcal{O}_{Z_k}(1)$. Taking $\ell$ sufficiently divisible, we may further assume $f$ is birational so that $\pi_k$ is also birational. Therefore we obtain 
\begin{equation*}
\vol(\mathcal{O}_Y(\frac{k}{\ell})) = \vol(\pi_k^*\mathcal{O}_Y(\frac{k}{\ell})) = \vol(\mathcal{O}_{Z_k}(1)).
\end{equation*}
The left hand side gives $k^n\vol(W)$ and $\vol(\mathcal{O}_{Z_k}(1)) = M_k^n$ so we conclude $\vol(W) = k^{-n}M_k^n$ in this case. 

In general, Theorem \ref{Fujita} reduces the proof to the case where $W$ is finitely generated. Let us explain it. For each $\ell$ set $W^{(\ell)}_{k}:=\mathrm{Im}[ S^{\frac{k}{\ell}}W_{\ell} \to W_{k}]$ if $k$ is devided by $\ell$ and otherwise set $W^{(\ell)}_{k}:=\{0\}$. Then $W^{(\ell)}$ defines a finitely generated graded linear series of $L$. Applying the above argument in the finitely generated case, we obtain $\vol(W^{(\ell)})=\ell^{-n}M_{\ell}^n$. Note that $\mu_{\ell}$ gives a resolution of $\mathfrak{b}(W^{(\ell)}_k)$ for any $k$ so that $M_{\ell}^{\otimes \frac{k}{\ell}} =\mu_{\ell}^*L^{\otimes k}\otimes\mathcal{O}(-\frac{k}{\ell}F_{\ell})$ corresponds to $W^{(\ell)}_k$. Then for any $\varepsilon>0$ sufficiently divisible $\ell$ assures 
\begin{equation*}
\vol(W) \geq \vol(W^{(\ell)})=\ell^{-n}M_{\ell}^n \geq \vol(W) -\varepsilon. 
\end{equation*} 
This ends the proof. 
\end{proof}
The reduction to the finitely generated case is the critical step to prove Theorem \ref{main} and the idea comes down to the proof of Theorem \ref{Fujita}. See also \cite{Ito12} for a similar argument. 

\begin{rem}\label{birational1}
In the above proof, the assumption $W$ is birational is crucial. For general $W$, $\pi_k$ possibly has degree greater than one and Proposiion \ref{algebraic part} does not hold. In fact in Example \ref{finite}, the right hand side in the proposition gives $\deg f$ times of $\vol(W)$. 
\end{rem}

\section{Proof of Theorem \ref{main}}
Let us prove the main theorem of this paper. In the sequel we fix the notation as in the statement of Theorem \ref{main}. Recall that $\varphi_k$ is a singular metric on $L$ defined by 
\begin{align*}
\varphi_k := \sup \bigg\{ \frac{1}{k} \log\abs{\sigma}^2  \ \bigg| \  \sigma \in W_k,  \ \sup_X \abs{\sigma}^2e^{-k\varphi} \leq 1. \bigg\}. 
\end{align*} 
By compactness of the unit ball in $W_k$, this defines a psh weight. We claim that $\varphi_k$ has algebraic singularities described by $\mathfrak{b}(W_k)$, the base ideal of $W_k$. To see this we will compare $\varphi_k$ with the corresponding Bergman kernel weight which is defined as follows. 
Fix a smooth volume form $dV$ on $X$. Set the $L^2$-norm by $\norm{\sigma}^2_{m\varphi}:= \int_X \abs{\sigma}^2e^{-m\varphi} dV$ and introduce the Bergman kernel weight $u_k$ as 
\begin{equation*}
u_k(x) := \sup \bigg\{ \frac{1}{k} \log\abs{\sigma(x)}^2  \ \bigg| \ \sigma \in W_k, \ \text{and}  \ \norm{\sigma}^2_{k\varphi} \leq 1. \bigg\}, 
\end{equation*}
for any $x \in X$. 
It is then easy to see that 
\begin{equation*}
(\limsup_{k \to \infty} u_k)^*= P_W\varphi 
\end{equation*}
holds. In fact, we obviously have 
\begin{equation*}
\norm{\sigma}^2_{k\varphi} \leq \int_X dV \cdot \sup_X \abs{\sigma}^2e^{-k\varphi}
\end{equation*}
and by the mean value inequality 
\begin{equation*}
\abs{\sigma(x)}^2 \leq C_r\sup_{B(x; r)}e^{k\varphi} \norm{\sigma}^2_{k\varphi}
\end{equation*}
holds for any sufficiently small ball $B(x; r)$ in a local coordinate. 

If one fix any orthonormal basis of $W_{k}$ with respect to the above $L^2$-norm, say $\{ \sigma_1, \dots \sigma_N \}$, then it is easy to see that $u_k =k^{-1}\log (\abs{\sigma_1}^2+\cdots+\abs{\sigma_N}^2)$ holds. Thus $\varphi_k$ also has algebraic singularities described by $\mathfrak{b}(W_k)$. 

Now by the definition of $F_k$ there exists a smooth semipositive $(1, 1)$-form $\gamma \in c_1(M_k)$ such that 
\begin{align*}
dd^c\mu_k^* u_{k} & = dd^c  \mu_k^* {k}^{-1}\log (\abs{\sigma_1}^2+\cdots+\abs{\sigma_N}^2) \\
& = {k}^{-1} (\gamma + [F_k])
\end{align*}
hold. Here $[F_k]$ stands for the current defined by the divisor $F_k$. 
Since $M_k$ is a globally generated line bundle we have $M_k^n=\int_X\gamma^n$ and the non-pluripolarity yields  $k^{-n}\gamma^n=\MA(\mu_k^*u_{k})$ on $X$. Therefore by Theorem \ref{comparison theorem} we obtain 

\begin{equation}\label{M_k}
k^{-n}M_k^n=\int_{X_k} \MA(\mu_k^*u_{k})=\int_{X} \MA(\varphi_k). 
\end{equation}

On the other hand, we have $P_W\varphi = (\sup_k \varphi_k )^*$, and the sequence $\varphi_k$ is essentially increasing, in the sense that $\varphi_k \leq \varphi_{\ell}$ if $k$ divides $\ell$. As a consequence, the sequence of $\psi_k:=\varphi_{2^k}$ is non-decreasing, and it converges in $L^1$-topology to $P_W\varphi$. We conclude using:
\begin{prop}\label{analytic part}
Let $\psi_k$ be an non-decreasing sequence of psh weights with small unbounded loci on a big line bundle $L$, and assume that $\psi_k \to \psi$ in the $L^1$ topology, {\em i.e.} $\psi = (\sup_k \psi_k)^*$. Then 
\begin{align*}
\int_X \MA(\psi_k) \to \int_X \MA(\psi). 
\end{align*}
\end{prop}
\begin{proof}
We have $\sup_k \int_X \MA(\psi_k) \leq \int_X \MA(\psi)$ by Theorem \ref{comparison theorem}. On the other hand, if $S$ is a proper algebraic subset of $X$ outside which $\psi_k$ is locally bounded, then $(dd^c\psi_k)^n \to (dd^c\psi)^n$ weakly on $X \setminus S$ by Theorem \ref{continuity properties of BT}, hence
\begin{align*}
\int_X \MA(\psi) \leq \liminf_{k \to \infty} \int_X \MA(\psi_k). 
\end{align*}
\end{proof}
Theorem \ref{main} is now concluded by Proposition \ref{algebraic part} together with Proposition \ref{analytic part}. 

\begin{rem}\label{birational2}
Theorem \ref{main} does not hold for general $W$ with $\vol(W)>0$. To be precise, taking a resolution of the base ideal of $W_k$, the right hand side in Theorem \ref{main} is given by $\lim_{k\to\infty} k^{-n}M_k^n$ just as (\ref{M_k}). As it was explained in Remark \ref{birational1}, however, Proposition \ref{algebraic part} is not true in general so that the equilibrium mass does not give $\vol(W)$. 
\end{rem} 

\section{Estimate of the Bergman kernels}

In this section we give a lower bound estimate of the Bergman kernels to prove  Theorem \ref{continuity}. 
Let us fix the notation as in Section $4$. For simplicity, we denote $P_W\varphi$ by $P\varphi$. 
The following is the required  one. 
\begin{prop}\label{key}
Fix a smooth volume form $dV$ on $X$. Set $\norm{\sigma}^2_{k\varphi}:= \int_X \abs{\sigma}^2e^{-k\varphi} dV$ and 
\begin{equation*}
u_k(x) := \sup \bigg\{ \frac{1}{k} \log\abs{\sigma(x)}^2  \ \bigg| \ \sigma \in W_k, \ \text{and}  \ \norm{\sigma}^2_{k\varphi} \leq 1. \bigg\}. 
\end{equation*}
for every $x \in X$. 
Assume that $W$ is finitely generated, birational, and further $\Proj W$, which is the image of the associated birational map $f_k\colon X \dashrightarrow \mathbb{P}W_{k}^*$ for any sufficiently divisible $k$, is normal. 
Then there exist a proper algebraic subset $S$ such that a subsequence of $u_k$ converges to $P\varphi$ uniformly on any compact set of $X \setminus S$. In particular, $P\varphi$ is continuous on a non-empty Zariski open set of $X$. 
\end{prop}
\begin{proof}
This has been already known in the complete case by \cite{Ber09} and can be extended to the present case by pushing the $L^2$-estimate to the image of the map $f_k\colon X \dashrightarrow \mathbb{P}W_k^*$. We claim for any compact set $K\Subset X \setminus S$ there exist a constant $C_K$ and a positive integer $\ell$ such that 
\begin{equation}\label{uniform}
P\varphi -\frac{C_K}{k} \leq u_k \leq P\varphi +\frac{C_K}{k}
\end{equation}
holds for any $k$ divisible by $\ell$. The right hand side inequality is obvious. The proof of the left hand side is essential and it needs some $L^2$-estimate for the solution of a $\bar{\partial}$-equation. 

Denote the image of the map $f_k\colon X \dashrightarrow \mathbb{P}W_k^*$ by $Y_k$. 
By the finite generation, the canonical map $S^kW_{\ell} \to W_{k\ell}$ is surjective for any $k$ and sufficiently large $\ell$. 
The restriction of $\mathbb{P}W_{k\ell}^*\hookrightarrow \mathbb{P}S^kW_{\ell}^*$ to $Y_{k\ell}$ and the restriction of Segre embedding $\mathbb{P}W_{\ell}^* \to \mathbb{P}S^kW_{\ell}^*$ to $Y_{\ell}$ are isomorphic. 
Thus we have $Y_{k\ell} \simeq Y_{\ell}$ for every $k\geq 1$. 
Let us fix $\ell$ and denote $Y_{\ell}$ by $Y$. Denote by $\mathcal{O}_Y(1)$ the restriction of $\mathcal{O}_{\mathbb{P}W_{\ell}^*}(1)$ to $Y$. 
Note that by the Kodaira vanishing theorem there exists a number $k_0$ such that $S^kW_{\ell}=H^0(\mathbb{P}W_{\ell}^*, \mathcal{O}(k)) \to H^0(Y, \mathcal{O}_Y(k))$ is surjective for any $k\geq k_0$. 
We take a modification $\mu\colon X_{\ell} \to X$ such that $X_{\ell}$ is smooth and there exist a decomposition $\mu^*\abs{W_{\ell}}=\abs{V}+E$ with a free linear system $V$ and an effective divisor $E$. Then the map $h\colon X_{\ell} \to Y$ defined by $V$ equals to $\mu\circ f$. 
Fix a standard section of $E$ and denote it by $s_E$. Given a section $s\in H^0(Y, \mathcal{O}_Y(k))$, one can push forward $h^*s\otimes s_E^{k}$ by $\mu$, thanks to the normality of $X$. We denote it by $f^*s$. This defines a map $f^*\colon H^0(Y, \mathcal{O}_Y(k)) \to W_{k\ell}$. 
Moreover, since $Y$ is normal by the assumption, given any $\sigma \in W_{k\ell}$ one can push-forward $\mu^*\sigma \otimes s_E^{-k}$ by the birational map $h$. We denote it by $h_*\sigma$. Thus we obtain an isomorphism $W_{k\ell}\simeq H^0(Y, \mathcal{O}_Y(k))$ for any $k$. 

Taking a resolution $\pi\colon \tilde{Y} \to Y$ we have a nef and big line bundle $\pi^*\mathcal{O}_Y(1)$ and an induced birational map $\tilde{f}\colon X \dashrightarrow \tilde{Y}$. The normality of $Y$ deduces $W_{k\ell}\simeq H^0(Y, \mathcal{O}_Y(k)) \simeq H^0(\tilde{Y}, \pi^*\mathcal{O}_Y(k))$. Here one can also pull-back or push-forward sections by the birational map $\tilde{f}$, taking a modification of the source space as above. 
Further, we can pull-back psh weights. That is, given a psh weight $\psi$ of a singular metric on $\pi^*\mathcal{O}_Y(1)$, we may push-forward $\tilde{h}^*\psi+\log \abs{s_E}^2$ by $\mu$ thanks to the normality of $X$. It defines a psh weight of a metric on $L^{\otimes \ell}$ and we denote it by $\tilde{f}^*\psi$. Finally, given a smooth weight $\varphi$ the push-forward $\tilde{f}_*P\varphi$ can be defined as follows: 
\begin{align*}
(\tilde{f}_*P\varphi)(y) := \sup{^*} \bigg\{ \frac{1}{k\ell} \log\abs{(\tilde{f}_*\sigma)(y)}^2  \ \bigg| \ k\geq 1, \ \sigma \in W_{k\ell}, \ \text{and}  \ \abs{\sigma}^2e^{-k\ell\varphi} \leq 1 \ \text{on} \ X \bigg\}. 
\end{align*}
This actually defines a psh weight of a singular metric on the $\mathbb{Q}$-line bundle $\ell^{-1}\pi^*\mathcal{O}_Y(1)$ and $\ell\tilde{f}_*P\varphi$ is a psh weight for the genuine line bundle $\pi^*\mathcal{O}_Y(1)$.

Fix a K\"{a}hler form $\omega$ on $\tilde{Y}$ and a psh weight $\psi$ of a metric on $\pi^*\mathcal{O}_Y(1)$, such that $\psi$ has algebraic singularity, $\tilde{f}^*\psi \leq \ell\varphi$, and $dd^c\psi\geq\omega$ hold. 
Let $S$ be a proper algebraic subset such that $\tilde{f}|_{X\setminus S}$ is isomorphic and $\psi$ is smooth outside $\tilde{f}(S) $. 

We claim that there exist sufficiently large $\ell$, $C$ and a section
$\sigma_k \in W_{k\ell}$ for each $k \geq 1$ such that
\begin{itemize}
 \setlength{\itemsep}{0pt}
  \item[$(1)$]
    $ \abs{\sigma_k(x)}^2e^{-k\ell P\varphi} \geq 
      C^{-1}  \ \ for \ any \ x \in K$, and   
  \item[$(2)$]
    $ \norm{\sigma_k}^2_{k\ell \varphi} \leq C $. 
\end{itemize} 
In fact, this implies
\begin{equation*}
 e^{k\ell u_{k\ell}}
      \geq \frac{\abs{\sigma_k(x)}^2}{\norm{\sigma_k}^2_{k\ell \varphi}} 
      \geq C^{-2}e^{k\ell P\varphi}. 
\end{equation*}
It then yields the inequality $P\varphi-\frac{C_K}{k\ell}\leq u_{k\ell}$ which is nothing but the left-hand side of (\ref{uniform}). 

Let us fix $x \in K$ and set $y:=\tilde{f}(x)$. Then applying the Ohsawa--Takegoshi $L^2$-extension theorem, for any $a\in\mathbb{C}$
one can get a holomorphic function $g$ on a small ball $B(y; r)$ such that $g(y)=a$ and  
\begin{equation*}
\int_{B(y; r)} \abs{g}^2 e^{-(k-1)\ell \tilde{f}_* P\varphi - \psi}dV_{\omega}
\leq C\abs{a}^2 e^{-(k-1)\ell \tilde{f}_* P\varphi - \psi} 
\end{equation*}
hold. Let $\rho$ be a cut-off function supported on $B(y; r)$. 
Solving the equation $\bar{\partial}v=\bar{\partial}(\rho g)$ by H\"{o}rmandar's $L^2$-method, we get a solution $v$ with 
\begin{equation*}
\int_{\tilde{Y}} \abs{v}^2 e^{-(k-1)\ell \tilde{f}_* P\varphi - \psi  - \rho (n+2)\log \abs{z}^2 }dV_{\omega}
\leq C\abs{a}^2 e^{-(k-1)\ell \tilde{f}_* P\varphi -  \psi}. 
\end{equation*}
Here we take a local coordinate $z=(z_1, \dots z_n)$ around $y$ and $\ell$ sufficiently large to ensure the positivity of the curvature of the total weight. Set $s:=\rho g -v \in H^0(\tilde{Y}, \pi^*\mathcal{O}_Y(k))$. Then it yields $s(y)=a$ and 
\begin{equation*}
\int_{\tilde{Y}} \abs{s}^2 e^{-(k-1)\ell \tilde{f}_* P\varphi - \psi}dV_{\omega}
\leq C\abs{a}^2 e^{-(k-1)\ell \tilde{f}_* P\varphi - \psi}. 
\end{equation*}
We may choose suitable $a$ so that the right hand side equals to $C$.  
Finally we set $\sigma_k:=\tilde{f}^*s$. 
Then it yields: 
\begin{itemize}
 \setlength{\itemsep}{0pt}
  \item[$(1)$]
    $\abs{\sigma_{k}(x)}^2e^{-(k-1)\ell P\varphi - \tilde{f}^*\psi} = 1$, 
  \item[$(2)$]
    $\norm{\sigma_k}^2_{(k-1)\ell P\varphi + \tilde{f}^*\psi} \leq C$. 
\end{itemize}
These inequalities respectively correspond to $(1)$ and $(2)$ of the claim. We indeed infer 
\begin{equation*}
 \norm{\sigma_k}^2_{k\ell \varphi} \leq \norm{s} ^2_{{(k-1)\ell P\varphi + \tilde{f}^*\psi}} \leq C. 
\end{equation*}
On the other hand we may assume $e^{(\ell P\varphi- \tilde{f}^*\psi)(x)} \leq C$ by the smoothness of $\psi$ around $y$ so that
\begin{equation*}
 1 = \abs{\sigma_{k}(x)}^2e^{-(k-1)\ell P\varphi -  \tilde{f}^*\psi} 
  \leq C\abs{\sigma_{k}(x)}^2e^{-k\ell P\varphi}
\end{equation*}
holds. Here $C$ depends on only $\ell$ and $K$. 
Therefore the claim has been shown to conclude the theorem. 
\end{proof}

 {\bf Acknowledgments.}
The author would like to express his gratitude 
to his advisor Professor Shigeharu Takayama for his warm encouragements, 
suggestions and reading the drafts. 
The author is grateful to Atsushi Ito for several helpful comments on the algebraic reduction arguments in this paper. 
The author greatly indebted to the referee for his many kind advices on many mathematical or linguistic points of improvement. In particular he indicated to the author a rather straightforward proof of Theorem \ref{main} after his reading the first version of this paper. In the first version the author used the uniform convergence in Proposition \ref{key} to derive Theorem \ref{main} but such a Bergman kernel estimate turned out to be unnecessary if one notes that $\varphi_k$ is essentially non-decreasing. 
This research is supported by JSPS Research Fellowships for Young Scientists (22-6742). 
 

\end{document}